\documentclass[12pt]{iopart}
\expandafter\let\csname equation*\endcsname\relax
\expandafter\let\csname endequation*\endcsname\relax

\usepackage{amsmath,amsfonts,amssymb,amsthm}
\usepackage{graphicx}
\usepackage{color,colordvi}
\usepackage{hyperref}
\usepackage{enumerate}
\usepackage{refcheck}

\def\softd{{\leavevmode\setbox1=\hbox{d}%
\hbox to 1.05\wd1{d\kern-0.4ex{\char039}\hss}}}
\def\softt{{\leavevmode\setbox1=\hbox{t}%
\hbox to \wd1{t\kern-0.6ex{\char039}\hss}}}

\newcommand{\D}{\mathrm{d}}

\newcommand{\R}{\mathbb{R}}

\newcommand{\OO}{\mathcal{O}}

\newtheorem{theorem}{Theorem}[section]
\newtheorem{lemma}{Lemma}[section]
\newtheorem{proposition}{Proposition}[section]

\newtheorem{remark}{Remark}[section]

\renewcommand{\theequation}{\arabic{section}.\arabic{equation}}

\newcommand{\defeq}{\vcentcolon=} 
\usepackage{mathtools}
\DeclarePairedDelimiter{\norm}{\lVert}{\rVert} 

\begin{document}

\title[Tunneling in soft waveguides:closing a book]
{Tunneling in soft waveguides: closing a book}

\author{Pavel Exner}
\address{Nuclear Physics Institute, Czech Academy of Sciences,
Hlavn\'{i} 130, \\ 25068 \v{R}e\v{z} near Prague, Czech Republic}
\address{Doppler Institute, Czech Technical University, B\v{r}ehov\'{a} 7, 11519 Prague, \\ Czech Republic}
\ead{exner@ujf.cas.cz}

\author{David Spitzkopf}
\address{Nuclear Physics Institute, Czech Academy of Sciences,
Hlavn\'{i} 130, \\ 25068 \v{R}e\v{z} near Prague, Czech Republic}
\address{Faculty of Mathematics and Physics, Charles University, V Hole\v{s}ovi\v{c}k\'ach 2, 18000 Prague \\ Czech Republic}
\ead{spitzkopf98@gmail.com}

\begin{abstract}
We investigate the spectrum of a soft quantum waveguide in two dimensions of the generalized `bookcover' shape, that is, Schr\"odinger operator with the potential in the form of a ditch consisting of a finite curved part and straight asymptotes which are parallel or almost parallel pointing in the same direction. We show how the eigenvalues accumulate when the angle between the asymptotes tends to zero. In case of parallel asymptotes the existence of a discrete spectrum depends on the ditch profile. We prove that it is absent in the weak-coupling case, on the other hand, it exists provided the transverse potential is strong enough. We also present a numerical example in which the critical strength can be assessed.
\end{abstract}


%
\vspace{2pc} \noindent{\it Keywords}: Quantum waveguides, Schr\"odinger operators, ditch-shaped potentials, discrete spectrum, tunneling effect

%
\submitto{\JPA}

%
%
%

\section{Introduction} 
\setcounter{equation}{0}

Quantum waveguides have been a subject of intense investigation in the previous three or four decades; for a survey (except the last few years) and an extensive bibliography we refer to \cite{EK15}. They serve as models of numerous physical systems, and at the same time, they represent a source of interesting mathematical problems. The nature of the interaction responsible for the guiding may differ. Often the walls of the waveguide are hard, meaning mathematically that the Hamiltonian is Dirichlet Laplacian in an appropriate spatial region. Alternatively singular Schr\"odinger operators are used with an attractive interaction supported by curves, surfaces, or more complicated sets of lower dimensionality in the configuration space.

Recently an alternative, more realistic model of `soft' waveguides attracted attention, cf.~\cite{Ex20} and further developments in \cite{KKK21, EL21, EKL22, EV23}, where the interaction term is a regular potential `ditch' the axis of which is a fixed curve. In such a description the tunneling between different parts of the guide is not suppressed while its transverse width need not be zero. The mentioned work dealt with the two-dimensional situation; there are extensions to three dimensions describing soft quantum tubes \cite{Ex22} and layers \cite{EKP20, KK23}, or even to situations where the ditch is replaced by an array of potential wells \cite{Ex23}.

The most remarkable feature of these systems are relations between their geometry and the energy spectrum, in particular, the fact that bending gives rise to an effective attractive interaction able to produce localized states in an infinitely extended ditch. This property is shared with the other waveguide types mentioned above, however, in contrast to them, the existence results obtained so far lack the universal character of their `hard' or `singular' counterparts. For two-dimensional Schr\"odinger operator with a $\delta$ interaction on a non-straight curve, for instance, Birman-Schwinger method yields the discrete spectrum existence under rather weak regularity and asymptotic straightness requirements \cite{EI01}, while the sufficient condition obtained by the same method in the soft case \cite{Ex20} is much weaker.

As an alternative, one may apply the variational method to the operator directly; the difficulty with this approach is that in general it is not easy to construct a suitable trial function. In \cite{KKK21} a simple case, usually dubbed \emph{bookcover} in the literature \cite{SM90}, was treated; the existence of bound states was proved for any bending angle $\theta\in(0,\pi)$ between the straight parts of the guide and any channel profile such that the transverse part of the operator has a discrete eigenvalue. Two questions have been left open, both related to the tunneling between the guide arms, namely the accumulation rate of the eigenvalues as the `book' closes, $\theta\to\pi-$, and the existence of the discrete spectrum in the limiting case, $\theta=\pi$.

A considerably more general existence result was obtained by variational method in \cite{EV23} allowing the finite non-straight part to be built over a $C^3$ smooth curve, while the semi-infinite asymptotic parts were supposed to be straight as before. The above two questions remained open again, and goal of this paper is to address them in the less restrictive geometrical setting of \cite{EV23} where the `spine' of the book may have an arbitrary shape provided it is finite and sufficiently regular. As in \cite{EL21} we characterize the ditch profile by a measure type potential which allows us to treat on the same footing regular and singular interactions; recall that for the latter the said two questions are also open.

In the next section we collect the necessary notions and hypotheses, we also find the essential spectrum which differs in the parallel and non-parallel case due to the different character of the tunneling in these two situations. Section~\ref{s:closing} is devoted to `closing' the book. Using the angle $\beta=\pi-\theta$ as a more natural parameter in this case, we show the curvature-induced eigenvalues fill the difference between the two essential spectra in the limit $\beta\to 0$, their number in any subinterval behaving as $\sim\beta^{-1}$.

After that we turn to the situation when the asymptotes are parallel, $\beta=0$. In Section~\ref{s:weak} we consider the weak-coupling case in which the channel profile is shallow; we show that the discrete spectrum is then empty. The opposite extreme is the topic of Section~\ref{s:strong}. The strong-coupling asymptotics offer more than one regime and we discuss three of them, a singular $\delta$ potential, a rectangular well of a non-flat bottom, and a sufficiently deep and narrow profile; in all of them we prove that strong coupling gives rise to discrete eigenvalues. As a byproduct, we generalize in Sec.~\ref{ss:dirichlet} the well-known result about the discrete spectrum existence \cite[Thm.~1.1]{EK15} to a class of Schr\"odinger operators in curved Dirichlet strips. As finding the critical interaction strength for \mbox{U-shaped} soft waveguides and their generalizations is out of reach analytically, we present in Section~\ref{s:num} a numerical example of an \mbox{U-shaped} channel with a specific transversal profile to illustrate bound states. Furthermore, we present computations regarding the dependence of the critical interaction strength on the geometry of the waveguide.

\section{Problem setting} 
\setcounter{equation}{0}

The object of our interest are two-dimensional Schr\"odinger operators with a potential in the shape of a `ditch' of a finite width and fixed profile. Let us describe the setting in more precise terms starting from the channel axis $\Gamma$. We suppose it is an infinite and smooth planar curve without self-intersections, naturally parametrized by its arc length~$s$, that is, the graph of a $C^1$ function $\Gamma:\:\R\to\R^2$ such that $\dot\Gamma(s)=1$, where the dot conventionally denotes the derivative with respect to $s$. With an abuse of notation we employ the same symbol for the map $\Gamma$ and for its range. If $\Gamma$ is $C^2$, its signed curvature is $\gamma:\: \gamma(s)= (\dot\Gamma_2\ddot\Gamma_1 - \dot\Gamma_1 \ddot\Gamma_2)(s)$ being naturally independent of the Cartesian coordinates used. The curve $\Gamma$ is supposed to satisfy the following assumption:
 \begin{enumerate}[(a)]
 \setlength{\itemsep}{0pt}
 \item $\Gamma$ is piecewise $C^2$-smooth and its curvature is of compact support. \label{assa}
 \end{enumerate}
Without loss of generality we may suppose that the straight parts of $\Gamma$ are the halflines $\Gamma_\pm:= \big\{\big(x,\pm(\rho+x\tan(\frac12\beta)\big):\, x\ge 0\big\}$ for a positive $\rho$ and $\beta\in[0,\frac12\pi)$, symmetric with respect to the $x$ axis, parametrized as
 \begin{equation} \label{paramext}
\Gamma_\pm:= \Big\{\big((\pm s-s_0)\cos(\textstyle{\frac12}\beta), \rho+(s\mp s_0)\sin(\textstyle{\frac12}\beta)\big):\, \pm s\ge s_0\Big\}
 \end{equation}
for some $s_0>2\rho$. The middle part of the curve referring to $s\in(-s_0,s_0)$ will be denoted as $\Gamma_\mathrm{int}$; without loss of generality again, we may suppose that it is contained in the open left halfplane, $x<0$. Recall that $\theta(s_2,s_1):= \int_{s_1}^{s_2} \gamma(s)\,\D s$ is the angle between the tangents at the points $\Gamma(s_j),\, j=1,2$, hence
 \begin{equation} \label{angle}
\beta = \pi - \int_{-s_0}^{s_0} \gamma(s)\,\D s\,;
 \end{equation}
we will be interested in the situations when $\beta=0$ and in the asymptotic regime $\beta\to 0+$.

Let us turn to the potential. We suppose that it is restricted transversally being a subset of the strip
 \begin{equation} \label{strip}
\Omega^a := \{ \textbf{x}\in\R^2:\:\mathrm{dist}(\textbf{x},\Gamma) < a \},
 \end{equation}
of a halfwidth $a>0$ built over $\Gamma$. We suppose that $a$ is small enough so that we can parametrize $\Omega^a$ using the parallel (often called Fermi) coordinates, given by the arc length $s$ and the distance $u$ along the normal $N(s) = (-\dot\Gamma_2(s),\dot\Gamma_1(s))$ to $\Gamma$ at the point~$s$,
 \begin{equation} \label{parstrip}
\textbf{x}(s,u) = \big(\Gamma_1(s)-u\dot\Gamma_2(s), \Gamma_2(s)+u\dot\Gamma_1(s) \big),
 \end{equation}
in such a way that
 \begin{enumerate}[(a)]
 \setcounter{enumi}{1}
 \setlength{\itemsep}{0pt}
\item the map $\R\times I_a\to \Omega^a$ defined by \eqref{parstrip} is a diffeomorphism, \label{assb}
 \end{enumerate}
where $I_a:=(-a,a)$; note that it can be satisfied, in particular, only if $a<\rho$.

The channel profile will be determined by the one-dimensional Schr\"odinger operator
 \begin{equation} \label{transop}
h_v = -\textstyle{\frac{\D^2}{\D x^2}}-v(x).
 \end{equation}
The potential $v$ is always supposed to be \emph{nontrivial}; when the interaction strength will be of importance we will replace $h_v$ by $h_{\lambda v}$ with $\lambda>0$. Mostly we will not strive
for the maximum generality focusing primarily on two situations; in the first
 \begin{enumerate}[(a)]
 \setcounter{enumi}{2}
 \setlength{\itemsep}{0pt}
 \item $v\in L^\infty(\R)$ with $\mathrm{supp}\,v\subset (-a,a)$ and $\inf\sigma(h_v)<0$\,; \label{assc}
 \end{enumerate}
needless to say one then has $D(h_v)=H^2(\R)$. Moreover, we have $\sigma_\mathrm{ess}(h_v)=[0,\infty)$ and $\epsilon_v:=\inf\sigma(h_v)<0$ is a simple eigenvalue. Using then the map \eqref{parstrip} we define
 \begin{subequations}
 \begin{align}
 V:\: \Omega^a & \to\R_+,\quad V(\textbf{x}(s,u))=v(u), \label{potential} \\[.5em]
 H_{\Gamma,V} & = -\Delta-V(\textbf{x}); \label{Hamiltonian}
 \end{align}
 \end{subequations}
in view of assumption \eqref{assc} the operator domain is $D(-\Delta)=H^2(\R^2)$. The second case of interest is the \emph{leaky curve} when instead of \eqref{Hamiltonian} we consider the operator
 \begin{enumerate}[(a)]
 \setcounter{enumi}{3}
 \setlength{\itemsep}{0pt}
 \item $H_{\Gamma,\alpha}$ associated with the quadratic form $\psi\mapsto \int_{\R^2} |\nabla\psi(\textbf{x})|^2 \D \textbf{x} - \alpha\int_\R |\psi(\Gamma(s))|^2 \D s$, where $\alpha>0$, defined on $H^1(\R^2)$.  \label{assd}
 \end{enumerate}
Alternatively one can say that the self-adjoint operator $H_{\Gamma,\alpha}$ acts as the negative Laplacian on functions $\psi\in H^1(\R^2) \cap H^2_\mathrm{loc}(\R^2\setminus\Gamma)$ which have at points of $\Gamma$ the jump of the normal derivative equal to $-\alpha\psi(\textbf{x})$; formally one can write $H_{\Gamma,\alpha}= -\Delta-\alpha\delta(\textbf{x}-\Gamma)$. It may be regarded as a singular version of the previous case, namely the family $H_{\Gamma,V_\varepsilon}$ corresponding to the scaled transverse potentials $v_\varepsilon(u)= \frac{1}{\varepsilon}v(\frac{u}{\varepsilon})$ converges in the norm resolvent sense to $H_{\Gamma,\alpha}$ with $\alpha:= \int_\R v(u)\,\D u\:$ as $\varepsilon\to 0\:$ \cite{EI01, BEHL17}. The analogue of the operator \eqref{transop}, formally $-\textstyle{\frac{\D^2}{\D x^2}}-\alpha\delta(x)$, will be denoted as $h_\alpha$; its ground state is in this case known explicitly, $\epsilon_\alpha:=\inf\sigma(h_\alpha)= -\frac14\alpha^2$.

While the two situations are of primary interest to us, they are particular cases of a more general channel profile. In the spirit of \cite{EL21} one can consider the operator
 \begin{enumerate}[(a)]
 \setcounter{enumi}{4}
 \setlength{\itemsep}{0pt}
 \item $H_{\Gamma,\mu}$ associated with the quadratic form $\psi\mapsto \int_{\R^2} |\nabla\psi(\textbf{x})|^2 \D \textbf{x} - \int_{\R^2} |\psi(\textbf{x})|^2 \D\mu(\textbf{x})$ defined on $H^1(\R^2)$, where $\D\mu(\textbf{x}) = (1-u\gamma(s))\,\D s\D\mu_\perp(u)$  corresponds to a finite positive Borel measure $\mu_\perp$ supported on the interval $(-a,a)$; \label{asse}
 \end{enumerate}
the assumptions \eqref{assc} and \eqref{assd} correspond respectively to $\D\mu_\perp(u)=v(u)\D u$ and $\D\mu_\perp(u)=\alpha\delta$, where $\delta$ is the Dirac measure supported at the origin. The analogue of \eqref{transop}, the self-adjoint operator associated with the form $f\mapsto \int_\R |f'(u)|^2 \D u- \int_\R |f(u)|^2 \D\mu_\perp(u)$ defined on $H^1(\R)$, will be denoted as $h_{\mu_\perp}$; by assumption we have $\epsilon_{\mu_\perp}:=\inf\sigma(h_{\mu_\perp})<0$.

In addition to \eqref{transop} we need also its double-well counterpart
 \begin{equation} \label{doubletrans}
h_{v,\rho} = -\textstyle{\frac{\D^2}{\D x^2}}-v(\rho+x)-v(-\rho-x)
 \end{equation}
and its analogues $h_{\alpha,\rho}$ and $h_{\mu_\perp,\rho}$ corresponding to assumptions \eqref{assd} and \eqref{asse}, respectively. The essential spectrum of these operators is, of course, $\R_+$ and their ground-state eigenvalues will be denoted as $\epsilon_{v,\rho}$ and analogously in the other two situations.

 \begin{proposition} \label{prop:spectess}
Assume \eqref{assa}, \eqref{assb} and \eqref{asse}, then $\sigma_\mathrm{ess}(H_{\Gamma,\mu}) = [\epsilon_{\mu_\perp}, \infty)$ holds provided $\beta>0$. If $\beta=0$, we have $\sigma_\mathrm{ess}(H_{\Gamma,\mu}) = [\epsilon_{\mu_\perp,\rho}, \infty)$ where the function $\rho\mapsto \epsilon_{\mu_\perp,\rho}$ is monotonously increasing with $\epsilon_{\mu_\perp,\infty} := \lim_{\rho\to\infty} \epsilon_{\mu_\perp,\rho} = \epsilon_{\mu_\perp}$.
 \end{proposition}
\begin{proof}
Under the assumptions \eqref{assc} or \eqref{assd} the essential spectrum was determined in \cite[Prop.~3.1]{Ex20} and \cite[Prop.~5.1]{EI01} for $\beta>0$, the case $\beta=0$ in the bookcover situation was dealt with in \cite[Prop.~2]{KKK21}. The more general waveguide satisfying the present assumptions can be treated similarly; the core of the argument is the existence of an increasing family of rectangles on which the potential channel is straight (or the parallel potential channels are straight for $\beta=0$) so that one can separate the variables and construct a suitable Weyl sequence combining transversally the ground-state eigenfunction of $h_{\mu_\perp}$ or $h_{\mu_\perp,\rho}$, restricted to intervals of a growing length, with a mollifier in the longitudinal direction. The interaction term in $h_{\mu_\perp,\rho}$ is mirror-symmetric, hence its ground state coincides with that of the single well and the Neumann boundary condition at the distance $\rho$ from its `center'. By the bracketing argument \cite[Sec.~XIII.15]{RS78} the function $\rho\mapsto \epsilon_{\mu_\perp,\rho}$  is non-decreasing. In fact, it is increasing with the indicated limit; one way to see that is to use the formula $\frac{\D}{\D\rho} \epsilon_{\mu_\perp,\rho} = -\varphi_{\mu_\perp,\rho}(0)^2\, \epsilon_{\mu_\perp,\rho}$, where $\varphi_{\mu_\perp,\rho}$ is the real-valued eigenfunction associated with $\epsilon_{\mu_\perp,\rho}$, obtained by an easy modification of the argument used in \cite{DH93} (an alternative is to employ an Agmon-type estimate which shows that the convergence is even exponentially fast).
\end{proof}

\section{Closing the book}\label{s:closing}
\setcounter{equation}{0}

Let us now investigate what happens if we are `closing the book', in other words, how the spectrum of $H_{\Gamma,\mu}$ behaves asymptotically in the limit $\beta \to 0$; our aim is to show that the eigenvalues will then fill the gap between the corresponding essential spectrum thresholds, $\epsilon_{\mu_\perp,\rho}$ and $\epsilon_{\mu_\perp}$. Specifically, we are going to show that for any $\nu \in (\epsilon_{\mu_\perp,\rho},\epsilon_{\mu_\perp})$ there are eigenvalues of $H_{\Gamma,\mu}$ in the interval $(\epsilon_{\mu_\perp,\rho},\nu)$ for $\beta$ small enough and to estimate their accumulation rate.
\begin{theorem} \label{theo:closingthebookBetaSmall}
Assume \eqref{assa}, \eqref{assb} and \eqref{asse}, then for any $\nu \in ( \epsilon_{\mu_\perp,\rho},\epsilon_{\mu_\perp})$ there is a constant $C_\nu>0$ such that $\dim E_{H_{\Gamma,\mu}}(\epsilon_{\mu_{\perp},\rho},\nu) \ge C_\nu \beta^{-1}$ holds for the corresponding spectral projection of $H_{\Gamma,\mu}$ provided that $\beta$ is small enough.
\end{theorem}
\begin{proof}
It is sufficient to find the appropriate number of linearly independent test functions for which the quadratic form associated with $H_{\Gamma,\mu}$ satisfies $\mathbf{q}[\phi] < \nu\|\phi\|^2$. We will construct such functions with the support in the part of the plane where $\Omega^a$ is straight, choosing them even with the respect to the axis of the angle between the two asymptotes, $\phi(x,y)=\phi(x,-y)$. In the upper halfplane we choose a pair of points of $\Gamma$ with the coordinates $s_1$, which can without loss of generality refer to the place where the straight parts begins, and $s_2>s_1$. The support of $\phi$ for $y\ge 0$ will be a semiinfinite strip $\Sigma$ the boundaries of which will be lines normal to $\Gamma$ at the points $s_1, s_2$; we divide it into the rectagular part $\Sigma_r$ and the wedge-shaped $\Sigma_w$ as sketched in Fig.~1.
\begin{figure}[h!]
\centering
    \includegraphics[clip, trim=6cm 19.2cm 6cm 4cm, angle=0, width=0.8\textwidth]{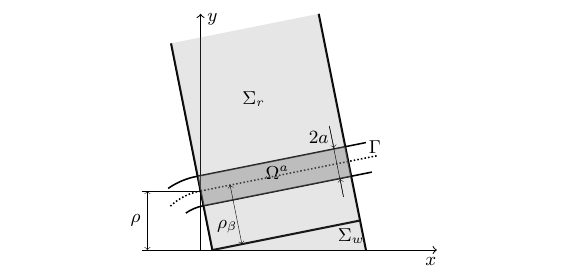}
\caption{The regions used in the proof of Theorem~\ref{theo:closingthebookBetaSmall}}
\end{figure}

It is obvious that in $\Sigma$ one can use the parallel coordinates; those referring to Cartesian $(x,y)$ will be denoted as $s(x,y)$ and $u(x,y)$. With this licence we use the following test function Ansatz,
\begin{equation}
\label{eq:test}
    \phi(x,y) \defeq \big[\chi_{\Sigma_r}(x,y) \varphi_{\rho_\beta} (u(x,y)+\rho_\beta) + \chi_{\Sigma_w}(x,y) \varphi_{\rho_\beta}(0)\big] f(s(x,y)), \;\; y\ge 0,
\end{equation}
where $\varphi_{\rho_{\beta}}$ is the eigenfunction associated with the eigenvalue $\epsilon_{\mu_{\perp},\rho_{\beta}}$ of operator \eqref{doubletrans} (for simplicity we drop the subscript $\mu_\perp$), $\rho_\beta = \rho\sec\frac{\beta}{2}$, and
\begin{equation}
\label{eq:DiscSpectrumDefFofX}
    f(s) \in C^2\left(\R \right) \quad\text{with}\quad \mathrm{supp\,}f \subset \left[s_{1},s_{2}\right],\;f\left(s_{1}\right) = f\left(s_{2}\right)=0\,;
\end{equation}
the functions \eqref{eq:test} are obviously admissible test functions belonging to $H^1(\R^2)$.

Let us now express the quadratic form starting from its kinetic term, $\|\nabla\phi\|^2$. Using the symmetry of $\phi$ and the fact that the parallel coordinates in $\Sigma$ are obtained by rotation of the Cartesian ones, we get
\begin{align}
\|\nabla\phi\|^2 & = 2 \int_\Sigma |\nabla_{s,u} \phi(x(s,u),y(s,u)) |^2\, \D s \D u \label{kinetic} \\[.3em]
& = 2 \int_{\Sigma_r} \big[ |f^{\prime}(s)|^2 |\varphi_{\rho_\beta}(u +\rho_{\beta}))|^2
+ |f(s)|^2 |\varphi_{\rho_\beta}^{\prime}(u +\rho_{\beta}))|^2\big]\, \D s \D u \nonumber \\[.3em]
& \quad + 2 \int_{\Sigma_w} |f^{\prime}(s)|^2 |\varphi_{\rho_\beta}(0)|^2\, \D s \D u \nonumber \\[.3em]
& = \|f'\|^2\|\varphi_{\rho_\beta}\|^2 + \|f\|^2\|\varphi'_{\rho_\beta}\|^2 + \|f'\|^2 |\varphi_{\rho_\beta}(0)|^2 L \tan\textstyle{\frac{\beta}{2}}, \nonumber
\end{align}
where $L:=s_2-s_1$, the norms refer to $L^2(s_1,s_2)$ and $L^2(\R)$, respectively, and  when integrating over $u$ we used the fact that $\varphi_{\rho_\beta}$ is an even function. For the interaction part of the quadratic form we get similarly
\begin{equation} \label{potpart}
    \|\phi\|^2_\mu  =  2 \int_{s_1}^{s_{2}} \int_{-a}^{a} |f(s)|^2 |\varphi_{\rho_\beta}(u)|^2\, \D\mu_{\perp}(u) \D s = \|f\|^2 \int_\R |\varphi_{\rho_\beta}(u)|^2\, \D\mu_{\perp}(u)
\end{equation}
and the norm of the test function is
\begin{equation} \label{testnorm}
    \|\phi\|^2 = \|f\|^2 \big(\|\varphi_{\rho_\beta}\|^2 + |\varphi_{\rho_\beta}(0)|^2 L \tan\textstyle{\frac{\beta}{2}}\big).
\end{equation}
Putting now \eqref{kinetic}--\eqref{testnorm} together, we can express the shifted quadratic form as
\begin{align*}
& \mathbf{q}[\phi] -\nu\|\phi\|^2 = \|f'\|^2 \Big[ \|\varphi_{\rho_\beta}\|^2 + |\varphi_{\rho_\beta}(0)|^2 L \tan\textstyle{\frac{\beta}{2}}\Big] \\[.3em]
& \qquad + \|f\|^2 \Big[ \|\varphi'_{\rho_\beta}\|^2 - \int_\R |\varphi_{\rho_\beta}(u)|^2\, \D\mu_{\perp}(u) -\nu \big(\|\varphi_{\rho_\beta}\|^2 + |\varphi_{\rho_\beta}(0)|^2 L \tan\textstyle{\frac{\beta}{2}}\big) \Big] \\[.3em]
& = \|f'\|^2 \Big[ \|\varphi_{\rho_\beta}\|^2 + |\varphi_{\rho_\beta}(0)|^2 L \tan\textstyle{\frac{\beta}{2}} \Big]
 + \|f\|^2 \Big[ \big(\epsilon_{\mu_\perp,\rho_\beta}-\nu \big) \|\varphi_{\rho_\beta}\|^2 -\nu |\varphi_{\rho_\beta}(0)|^2 L \tan\textstyle{\frac{\beta}{2}}\big) \Big],
\end{align*}
where we have used the fact that $\varphi_{\rho_\beta}$ is by assumption the ground-state eigenfunction of $h_{\mu_\perp,\rho_\beta}$. The last expression is negative provided
\begin{equation}
\label{eq:negform}
    \frac{\|f'\|^2}{\|f\|^2} < \frac{\nu - \epsilon_{\mu_\perp,\rho_\beta} + \nu\eta^2_{\rho_\beta} L \tan\textstyle{\frac{\beta}{2}}}{1 + \eta^2_{\rho_\beta} L \tan\textstyle{\frac{\beta}{2}}},
\end{equation}
where we set $\eta_{\rho_\beta}:= \frac{|\varphi_{\rho_\beta}(0)|}{\|\varphi_{\rho_\beta}\|}$ for the sake of brevity; this quantity is bounded as a function of the angle $\beta$ and $\lim_{\beta\to 0} \eta_{\rho_\beta} = \eta_\rho >0$. The right-hand side of inequality \eqref{eq:negform}, which we denote as $R_\nu(L,\beta)$, is continuous with respect to the parameters and for small enough $\beta$ it is positive tending to $\nu - \epsilon_{\mu_\perp,\rho}$ as $\beta\to 0$.

The left-hand side is in view of \eqref{eq:DiscSpectrumDefFofX} nothing else than the normalized quadratic form of the Dirichlet Laplacian on an interval of length $L$, hence the maximum number $n_\nu$ of mutually orthogonal trial functions making the form $\mathbf{q}[\cdot]-\nu\|\cdot\|^2$ negative has to satisfy the inequality $\big(\frac{\pi n_\nu}{L}\big)^2 < R_\nu(L,\beta)$. The indicated properties of the right-hand side imply, in particular, that there are positive $a_\nu,\,b_\nu$ such that $R_\nu(L,\beta)>a_\nu$ holds whenever $L\tan\frac{\beta}{2} < b_\nu$, and consequently, we get for $n_\nu$ the bound
$$
n_\nu < \frac{L}{\pi}\sqrt{a_\nu} < \frac{\sqrt{a_\nu}}{\pi}\,\frac{b_\nu}{\tan\frac{\beta}{2}} < \frac{2b_\nu\sqrt{a_\nu}}{\pi}\,\beta^{-1},
$$
which is the result we have set out to prove.
\end{proof}

\section{Parallel asymptotes, weak coupling}\label{s:weak}
\setcounter{equation}{0}

Let us now turn to the situation when the `book is closed', $\beta=0$. We know from Proposition \ref{prop:spectess} how the essential spectrum looks in this case; what we want to know is whether the discrete spectrum is nonempty. It appear that the answer depends on the strength of the transversal coupling. To this aim we introduce a coupling constant $\lambda>0$ into the picture; modifying assumption~\eqref{asse} we consider operator $H_{\Gamma,\lambda\mu}$ associated with the quadratic form $\psi\mapsto \int_{\R^2} |\nabla\psi(\textbf{x})|^2 \D \textbf{x} - \lambda \int_{\R^2} |\psi(\textbf{x})|^2 \D\mu(\textbf{x})$ defined on $H^1(\R^2)$. First we focus on the situation when the coupling is weak.
\begin{theorem} \label{thm:weakcoupl}
Assume \eqref{assa}, \eqref{assb} and \eqref{asse}. Then $\sigma_{disc}\left(H_{\Gamma,\lambda\mu}\right)$ is empty for all $\lambda$ small enough.
\end{theorem}
We will need the following claim extending the well-known result of \cite{Si76}:
\begin{lemma} \label{l:weak1Dmeasure}
Let $h_{\lambda\eta}$ be the Schr\"odinger operator in $L^2(\R)$ associated with the quadratic form $f\mapsto \int_\R |f'(u)|^2 \D u - \lambda\int_\R |f(u)|^2 \D\eta(u)$ defined on $H^1(\R)$, where $\eta$ is a compactly supported finite positive Borel measure. For all sufficiently small positive $\lambda$, the operator has exactly one eigenvalue $\epsilon(\lambda)$ which behaves asymptotically as
\begin{equation}
\label{weak1D}
\sqrt{-\epsilon(\lambda)} = \frac12\lambda\eta(\R) + \OO(\lambda^2).
\end{equation}
\end{lemma}
\begin{proof}
As in the regular case, on can use the Birman-Schwinger principle: by \cite[Lemma~2.3]{BEKS94} we have to check that the operator $\lambda R_{\eta\eta}(\kappa)$, the trace of the resolvent $\big(-\frac{\D^2}{\D x^2}+\kappa^2\big)^{-1}$ at the space $L^2(\R,\D\eta)$, has for small $\lambda$ a simple eigenvalue one. The trace acts on $f\in L^2(\R,\D\eta)$ as $R_{\eta\eta}(\kappa)f = G_\kappa \star f$ $\,\eta$-almost everywhere with $G_\kappa = \frac{1}{2\kappa} \e^{-\kappa|\cdot|}$ defines the kernel of the one-dimensional Laplacian resolvent. As in \cite{Si76} we split the operator into two parts, $R(\kappa)=L_\kappa+M_\kappa$, where $L_\kappa$ is the integral operator with the kernel $\frac{1}{2\kappa}$ whose range is one-dimensional consisting of constant functions, and $M_\kappa$ with the kernel $\frac{1}{2\kappa} \big( \e^{-\kappa|x-y|}-1 \big)$. The last opertor is obviously bounded which means that $I-\lambda M_\kappa$ is invertible for $\lambda$ small enough; using then the identity
$$
I-\lambda R(\kappa) = (I-\lambda M_\kappa) \big( I- (I-\lambda M_\kappa)^{-1} \lambda L_\kappa\big),
$$
we see that the task is reduced to finding $\mathrm{Ker} \big( I- (I-\lambda M_\kappa)^{-1} \lambda L_\kappa\big)_{\eta\eta}$, and since we have noted that the operator is one-dimensional, the explicit form of $L_\kappa$ yields the equation
$$
\kappa \eta(\R) = \frac{\lambda}{2} \int_{\R^2} (I-\lambda M_\kappa)^{-1}(x,y)\, \D\eta(x)\,\D\eta(y).
$$
The rest of the proof proceeds as in the regular case; to complete the proof one has to check that the right-hand side is a real analytic function of $\lambda$ up to the point $\lambda=0$ and to expand it there in the leading order.
\end{proof}

\begin{proof}[Proof of Theorem~\ref{thm:weakcoupl}]
We use again a bracketing argument and impose additional Neumann boundary conditions on the $y$-axis obtaining thus the inequality
\begin{equation}
\label{eq:Nbrack}
H_{\Gamma,\lambda \mu} \geq H_{\Gamma,\lambda \mu}^\mathrm{c} \oplus H_{\Gamma,\lambda \mu}^\mathrm{a},
\end{equation}
where the first operator in the direct sum refers to the curved part in the left halfplane and the second to the parallel channels in the right one. It is thus sufficient to check that none of the two operators has spectral points below $\epsilon_{\lambda\mu_\perp,\rho}$. The spectrum of $H_{\Gamma,\lambda \mu}^\mathrm{a}$ is purely absolutely continuous covering the interval $[\epsilon_{\lambda\mu_\perp,\rho},\infty)$, and choosing for $\eta$ in Lemma~\ref{l:weak1Dmeasure} the `double-well measure', $\D\eta(u)=\D\mu_\perp(\rho+u)+ \D\mu_\perp(-\rho-u)$, we see that
\begin{equation}
\label{eq:right_est}
\inf\sigma(H_{\Gamma,\lambda \mu}^\mathrm{a}) = \epsilon_{\lambda\mu_\perp,\rho} = \lambda^2 \mu_\perp(\R)^2 + \OO(\lambda^3) \quad\text{as}\;\; \lambda\to 0.
\end{equation}

Consider next the operator $H_{\Gamma,\lambda \mu}^\mathrm{c}$ and compare it with $H_{\tilde\Gamma,\lambda \mu}$ where $\tilde\Gamma$ a closed curve being a union of $\Gamma_{int}$ and its mirror image with respect to the $y$-axis. By construction $\tilde\Gamma$ is $C^1$-smooth, and since it is a finite loop, the negative spectrum of $H_{\tilde\Gamma,\lambda \mu}$ is discrete. In view of the symmetry, the operator allows for a parity decomposition in the $x$ direction, $H_{\tilde\Gamma,\lambda \mu} = H_{\tilde\Gamma,\lambda \mu}^\mathrm{sym} \oplus H_{\tilde\Gamma,\lambda \mu}^\mathrm{asym}$ and it holds $H_{\tilde\Gamma,\lambda \mu}^\mathrm{sym} \le H_{\tilde\Gamma,\lambda \mu}^\mathrm{asym}$, in particular, $\inf\sigma(H_{\tilde\Gamma,\lambda \mu}) = \inf\sigma(H_{\tilde\Gamma,\lambda \mu}^\mathrm{sym}) = \inf\sigma(H_{\Gamma,\lambda \mu}^\mathrm{c})$. By assumption, the measure $\mu_\perp$ is positive, and the same is, of course, true for $\mu$. This allows us to use the result of \cite{KL14} by which $H_{\tilde\Gamma,\lambda \mu}$ has for all $\lambda$ small enough a single eigenvalues which behaves asymptotically as follows,
\begin{equation}
\label{eq:loop_est}
\epsilon_0(\lambda) = -\big(C_\mu+o(1)\big) \exp\Big(-\frac{4\pi}{\lambda\mu(\R^2)}\Big) \quad\text{as}\;\; \lambda\to 0,
\end{equation}
where $C_\mu$ is a positive constant depending on the measure $\mu$. A comparison of \eqref{eq:right_est} and \eqref{eq:loop_est}, taking into account that $\mu_\perp(\R)>0$, implies in view of \eqref{eq:Nbrack} that for sufficiently small $\lambda$ the spectrum of $H_{\Gamma,\lambda \mu}$ below $\epsilon_{\lambda\mu_\perp,\rho}$ is empty; this concludes the proof.
\end{proof}
 \begin{remark} \label{rem:sign_change}
{\rm The assumption \eqref{asse} includes positivity of the measure. This may not be necessary, for instance, a regular potential of assumption \eqref{assc} might be sign-changing. The claim of Theorem~\ref{thm:weakcoupl} remains nevertheless valid in such a situation even if $\int_\R v(x)\,\D x=0$. Indeed, we know from \cite{Si76} that the weakly coupled state then still exists and it satisfies $\sqrt{-\epsilon(\lambda)} = \frac14\lambda^2\int_{\R^2} v(x)|x-y|v(y)\, \D x\,\D y + \OO(\lambda^3)$, hence the essential spectrum threshold behaves as $\OO(\lambda^4)$. We cannot apply the result of \cite{KL14} directly to get \eqref{eq:loop_est} because the asymptotics was obtained under the positivity assumption -- see, however, Remark~3.5 of the paper -- it is easy to bypass this limitation. If $\mu_\perp$ is sign-changing, we can write it as a difference of two positive measures with disjoint supports, $\mu_\perp = \mu_\perp^{(+)} - \mu_\perp^{(-)}$, and the same decomposition applies to $\mu$. Since $H_{\tilde\Gamma,\lambda \mu} \ge H_{\tilde\Gamma,\lambda \mu^{(+)}}$, the eigenvalue $\epsilon_0(\lambda)$ is bounded from below by the right-hand side of \eqref{eq:loop_est} with $\mu(\R^2)$ replaced by $\mu^{(+)}(\R^2)$ tending exponentially fast to zero, so for small enough $\lambda$ it is again the essential spectrum threshold which dominates. }
 \end{remark}

\section{Parallel asymptotes, strong coupling}\label{s:strong}
\setcounter{equation}{0}

In contrast to the weak coupling, the strong one offers a wider variety of asymptotic regimes referring to particular subclasses of potentials. In this section we focus on three of them.

\subsection{Leaky curves}\label{ss:leaky}

Consider first a `leaky curve', that is, operator $H_{\Gamma,\alpha}$ with the attractive $\delta$ interaction supported by the curve $\Gamma$ with parallel asymptotes, $\beta=0$.

 \begin{proposition} \label{thm:leaky}
Assume \eqref{assa} and \eqref{assd}. If, in addition, $\Gamma\in C^4(\R)$, then $\sigma_\mathrm{disc}(H_{\Gamma,\alpha}) \ne \emptyset$ holds for all $\alpha$ large enough. Moreover, the number of eigenvalues (with the multiplicity taken into account) does not exceed that of the Schr\"odinger operator $S_\Gamma= -\frac{\D^2}{\D s^2} - \frac14\gamma(s)^2$ on $L^2(\R)$, the bound being saturated as $\alpha\to\infty$.
 \end{proposition}

\begin{proof}
The claim follows from the strong-coupling asymptotics of $\sigma_\mathrm{disc}(H_{\Gamma,\alpha})$ derived in \cite{EY01}, see also \cite[Cor.~10.3.1]{EK15}. It was stated there for curves the asymptotes of which were not parallel, however, this assumption was not used in its proof; the vital assumptions are satisfied since $\Gamma\in C^4(\R)$ and its curvature is compactly supported. The $j$th eigenvalues of $H_{\Gamma,\alpha}$ then behaves asymptotically as
 \begin{equation} \label{leakyasympt}
\epsilon_j(\alpha) = -\frac14\alpha^2 + \epsilon_j + \mathcal{O}(\alpha^{-1}\ln\alpha) \quad\text{for}\;\; \alpha\to\infty,
 \end{equation}
where $\epsilon_j<0$ is the $j$th eigenvalue of the comparison operator $S_\Gamma$. The curvature-induced potential of $S_\Gamma$ is nonzero, attractive, bounded and compactly supported, hence $\sigma_\mathrm{disc}(S_\Gamma)$ is nonempty and finite.
\end{proof}

\subsection{An interlude: Dirichlet guides}\label{ss:dirichlet}

To consider another strong-coupling asymptotic, we shall extend a known result about the discrete spectrum of curved Dirichlet strips \cite[Thm.~1.1]{EK15}. Comparing the assumption concerning the Dirichlet strips in that book with those about the potential support used here, we see that \emph{(i)} of \cite[Sec.~1.1]{EK15} guaranteeing the existence of parallel coordinates coincides with \eqref{assb}, while \eqref{assa} is weaker than \emph{(ii)} because it requires the curvature only to be piecewise continuous. For the needs of the present paper it would be sufficient to have the curvature compactly supported which would imply assumptions \emph{(iii)}--\emph{(v)} but we will not limit ourselves here to this case and prove the result for strips $\Omega^a$ which are only asymptotically straight.

The object to consider is a Dirichlet waveguide with a generally non-flat bottom, that is, the operator $H_{\mathrm{D},v}:= -\Delta_\mathrm{D}^{\Omega^a} - V(x)$, where $V(x)$ is the potential \eqref{potential} corresponding to $v$ satisfying assumption \eqref{assc}. Let $h_{\mathrm{D},v}$ be the operator on $L^2(I_a)$ with the domain $H^2(I_a)\cap H_0^1(I_a)$ acting as \eqref{transop}, in other words, the transverse Schr\"odinger operator on $I_a$ with Dirichlet boundary conditions at $\pm a$. The spectrum of this operator is simple and purely discrete; we denote $\epsilon_{\mathrm{D},v}:= \inf\sigma(h_{\mathrm{D},v})$ and $\chi_0$ will be the eigenfunction corresponding to this eigenvalue.

\begin{theorem}
\label{theo:ExtensionDirichletGuides}
Let $\Omega^a$ satisfy assumptions \eqref{assa} and \eqref{assb}, and $v$ satisfy assumption \eqref{assc}. Let further $\lim_{|s|\to\infty} \gamma(s)=0$, then the operator $H_{\mathrm{D},v}$ has at least one eigenvalue below $\inf\sigma_\mathrm{ess}(H_{\mathrm{D},v})= \epsilon_{\mathrm{D},v}$ unless $\gamma=0$ identically.
\end{theorem}
\begin{proof}
The argument follows the proof of the mentioned Theorem~1.1 in \cite{EK15} except that we weaken the regularity requirements on $\Gamma$ in the spirit of \cite{KS12}. The main tool is the quadratic form of the operator $H_{\mathrm{D},v}-\epsilon_{\mathrm{D},v}I$ which can be, using the parallel coordinates, written as
\begin{equation}
\label{eq:HQWExtendedQuadraticForm}
\mathbf{q}\left[\psi \right] \defeq \norm{g^{-1/4} \partial_s \psi}^2 +\norm{g^{1/4} \partial_u \psi}^2 - \norm{g^{1/4} V\psi}^2 - \epsilon_{\mathrm{D},v} \norm{g^{1/4} \psi}^2
\end{equation}
for any $\psi\in H^1(\R\times I_a)$; we recall that $\sqrt{g(s,u)}=(1-u\gamma(s))$. To find the essential spectrum threshold, we use Neumann bracketing, adding Neumann condition at the segments perpendicular to $\Gamma$ at $s=\pm s_0$. This allows us to estimates $H_{\mathrm{D},v}$ from below by the direct sum $H_{\mathrm{D},v}^{(-)} \oplus H_{\mathrm{D},v}^{(0)} \oplus H_{\mathrm{D},v}^{(+)}$. The middle part refers to a compact region being thus irrelevant from the point of the essential spectrum. The shifted quadratic forms of the other two parts can be estimated from below as follows,
$$
\mathbf{q}^{(\pm)}\left[\psi \right] \ge (1-a\|\gamma_\pm\|_\infty)^{1/2}\big[\norm{\partial_s \psi}^2_\pm +\norm{\partial_u \psi}^2_\pm - \norm{V\psi}^2_\pm \big] - \epsilon_{\mathrm{D},v} (1+a\|\gamma_\pm\|_\infty)^{1/2} \norm{\psi}^2_\pm,
$$
where the norms refer to $L^2((s_0,\infty)\times I_a)$ and $L^2((-\infty,-s_0)\times I_a)$, respectively, and $\gamma_\pm$ are the restrictions of $\gamma$ to the appropriate intervals, and we have used the inequality $(1-a\|\gamma_\pm\|_\infty)^{1/2} \le (1+a\|\gamma_\pm\|_\infty)^{-1/2}$. The parts in the square bracket are nothing but quadratic forms of the operators on the respective parts of the straight channel and as such they they bounded from below by $\epsilon_{\mathrm{D},v} \norm{\psi}^2_\pm$, so we have
$$
\mathbf{q}^{(\pm)}\left[\psi \right] \ge - \epsilon_{\mathrm{D},v} \big[(1+a\|\gamma_\pm\|_\infty)^{1/2} - (1-a\|\gamma_\pm\|_\infty)^{1/2}\big] \norm{\psi}^2_\pm,
$$
and since $\gamma(s)$ tends to zero as $|s|\to\infty$ by assumption, the left-hand side can be made arbitrarily close to zero by choosing $s_0$ large enough.

The prove the second claim, one has to find a trial function from the form domain of $H_{\mathrm{D},v}$ which makes the left-hand side of \eqref{eq:HQWExtendedQuadraticForm} negative. We use the idea proposed in this context in \cite{GJ92} and choose the function in the form $\psi \defeq \phi_{\varsigma} \chi_0 + \varepsilon f$, where $f$ will be chose later and $\phi_\varsigma$ is a suitable mollifier, say
\begin{equation} \label{mollif}
 \phi_{\varsigma} \defeq \begin{cases*}
 \phi(s) & for $ |s| \le s_1$ \\
\phi \left(s_0\,\mathrm{sgn}s + \varsigma(s-s_0\,\mathrm{sgn}s) \right)& for $|s| > s_1$
\end{cases*},
\end{equation}
with $\phi \in C_0^{\infty}(\R)$ such that $\phi(s)=1$ holds for $s \in [-s_1, s_1]$ for some $s_1>0$. We have
\begin{align}
\label{eq:HQWExtendedFive}
& \mathbf{q}\left[\phi_\varsigma \chi_0\right] =  \int_{\mathbb{R}} \langle g^{-1/2} \rangle(s) |\phi'_\varsigma(s)|^2 \D s + \iint g^{1/2}(s,u) |\chi'_0(u)|^2 |\phi_\varsigma(s)|^2 \D s\D u \\ \nonumber
& \quad - \iint  g^{1/2}(s,u) v(u) |\chi_0(u)|^2 |\phi_\varsigma(s)|^2 \D s\D u - \epsilon_{\mathrm{D},v} \iint  g^{1/2}(s,u) |\chi_0(u)|^2 |\phi_\varsigma(s)|^2 \D s\D u,
\end{align}
where $\langle g^{-1/2} \rangle(s) = \int_{-a}^a g^{-1/2}(s,u)\,\D u$ is the transverse average of the inverse Jacobian. By Fubini theorem the order of integration in the last three terms makes no difference. Eigenfunctions of $h_{\mathrm{D},v}$ may be chosen real and by assumption we have
$$
-\chi_0''(u) -v(u)\chi_0(u) = \epsilon_{\mathrm{D},v}\chi_0(u)\,;
$$
an easy integration by parts with $(1-u\gamma(s))\chi_0(u)\D u$ using the fact, that $\chi_0(\pm a)=0$, shows that the sum of the three terms is zero. The remaining term is easily estimated; we get the bound
\begin{equation}
\label{eq:HQWExtendedFirstPartLeq}
 \mathbf{q}\left[\phi_\varsigma \chi_0\right] \leq \frac{\varsigma}{(1- a \norm{\gamma}_{\infty})^{1/2}} \norm{\phi'}^2
\end{equation}
showing that by the choice of $\varsigma$ the positive contribution from the trial function tails can be made arbitrarily small.

To complete the construction, we have to choose the function $f$. We pick it from $C_0^\infty((-s_1,s_1)\times I_a)$ in which case $\psi\in H^1(\R\times I_a)$ and
\begin{equation}
\label{GJ}
 \mathbf{q}[\psi] = \mathbf{q}\left[\phi_\varsigma \chi_0\right] +2\varepsilon\, \mathrm{Re}\, \mathbf{q}(\phi_\varsigma \chi_0,f) +\varepsilon^2 \mathbf{q}[f],
\end{equation}
where $\mathbf{q}(\cdot,\cdot)$ is the corresponding sesquilinear form. Using the fact that the supports of $\phi'_\varsigma\chi_0$ and $f$ are disjoint, we evaluate easily the linear term coefficient,
$$
2\,\mathrm{Re}\, \mathbf{q}(\phi_\varsigma \chi_0,f) = \big( g^{1/2}(\chi'_0 -V\chi_0 - \epsilon_{\mathrm{D},v}\chi_0),f\big).
$$
Since $f$ is picked from an infinitely dimensional space, we can obviously choose it in such a way that the coefficient is nonzero and, say, negative. For small $\varepsilon>0$ the linear term dominates over the quadratic one; then we can choose $\varepsilon$ so that the sum of the last two terms on the right-hand side of \eqref{GJ} is negative, and subsequently to choose $\varsigma$ in order not to spoil the negativity of the whole estimating expression.
\end{proof}

Note that we have identified only the threshold of $\sigma_\mathrm{ess}(H_{\mathrm{D},v})$. Under stronger regularity assumptions one can prove that the essential spectrum covers the whole interval $[\epsilon_{\mathrm{D},v},\infty)$ but we will not need it here.

\subsection{Making the ditch deeper}\label{ss:deep}

Let us return to our problem in the situation with parallel asymptotes, $\beta=0$. A natural way to achieve strong coupling for a regular potential channel consists of modifying its depth without changing the profile of its bottom. For the sake of brevity we denote by $\chi_a$ the characteristic function of the set $\Omega^a\subset\R^2$, with an abuse of notation we will use the same symbol the characteristic function of the interval $I_a\subset\R$ and for the measure $\chi_a(x)\D x$ on $\R$; then we consider the behavior of the operator $H_{\Gamma,V+\lambda\chi_a}$.
\begin{theorem}
\label{theo:deepditch}
Let assumptions \eqref{assa}, \eqref{assb} and \eqref{assc} be satisfied, then we have
\begin{equation}
\label{strong_ess}
\inf\sigma_\mathrm{ess}(H_{\Gamma,V+\lambda\chi_a}) = \epsilon_{\mu_\perp + \lambda\chi_a,\rho} = -\lambda + \epsilon_{\mathrm{D},v} + \OO\big(\e^{-c\sqrt{\lambda}}\,\big)
\end{equation}
as $\lambda\to\infty$ for some $c>0$, and $\sigma_\mathrm{disc}(H_{\Gamma,V+\lambda\chi_a})\ne\emptyset$ for all $\lambda$ large enough.
\end{theorem}
\begin{proof}
Consider the operator family $\{H_{\Gamma,V+\lambda\chi_a}+\lambda I:\: \lambda\ge 0\}$. The corresponding family of quadratic forms is monotonously increasing, then it follows from \cite[Thm.~S.14]{RS80} that these operators converge to $H_{\mathrm{D},V}$ as $\lambda\to\infty$ in the generalized strong resolvent sense \cite[Sec.~9.3]{W00}. By the same argument, $h_{\Gamma,v+\lambda\chi_a}+\lambda I \to h_{\mathrm{D},v}$ as $\lambda\to\infty$ in the generalized strong resolvent sense. We will use the last result to prove \eqref{strong_ess}. Its first relation follows form Proposition~\ref{prop:spectess}. To get the second one, we note that by the standard double-well estimate \cite[Thm.~1.5]{Si84} we have $\epsilon_{\mu_\perp + \lambda\chi_a,\rho} = \epsilon_{\mu_\perp + \lambda\chi_a} + \OO\big(\e^{-c\sqrt{\lambda}}\,\big)$ where the constant $c$ is given by the associated Agmon metric.

The second claim follows from Theorem~\ref{theo:ExtensionDirichletGuides}. As the curvature $\gamma$ is by assumption nonzero and compactly supported, the assumptions are satisfied and the limiting operator $H_{\mathrm{D},V}$ has at least one eigenvalue below $\epsilon_{\mathrm{D},v}$, hence $H_{\Gamma,V+\lambda\chi_a}+\lambda I$ has an eigenvalue below $\epsilon_{\mu_\perp + \lambda\chi_a,\rho}$, and consequently, below $-\lambda + \epsilon_{\mathrm{D},v}$ for $\lambda$ large enough.
\end{proof}

\subsection{Scaling with increasing volume}\label{ss:scaling}

Another strong coupling situation, extending the result of Proposition~\ref{thm:leaky} to regular potentials, arises if we combined linear scaling with a sufficiently fast increasing mean strength, more specifically, if we consider the following family of operators,
\begin{equation}
\label{scaledop}
H_{\Gamma,V_{g(\lambda)}} \quad\text{with}\quad v_{g(\lambda)}(x):= g(\lambda) v(\lambda x),
\end{equation}
with $g$ being a suitable function such that $\lim_{\lambda\to\infty} \frac{g(\lambda)}{\lambda}=\infty$ as $\lambda\to\infty$.
 \begin{proposition} \label{thm:mix}
Adopt assumptions \eqref{assa}--\eqref{assc} and suppose, in addition, that $\Gamma\in C^4(\R)$. Then there is a function $g_0$ with the indicated properties such that $\sigma_\mathrm{disc}(H_{\Gamma,V_{g(\lambda)}})$ is nonempty for any $g\ge g_0$ and all $\lambda$ large enough.
 \end{proposition}
\begin{proof}
Instead of \eqref{scaledop}, consider the two-parameter family of operators $H_{\Gamma,V_{\xi,\lambda}}$ with the profile potential $v_{\xi,\lambda}(x):= \xi\lambda v(\lambda x)$. As we have already mentioned, for a fixed $\xi>0$ these operators converge in the norm-resolvent sense to $H_{\Gamma,\alpha_\xi}$ with $\alpha_\xi:=\xi\|v\|_1$ as $\lambda\to\infty$, cf.~\cite{EI01, BEHL17}. The difficulty to deal with is that the essential spectrum threshold depend on $\xi$. Given a below bounded self-adjoint operator $A$, we denote by $\mu_k(A)$ the corresponding numbers obtained from the minimax principle\footnote{This is a traditional notation and we are sure there is no danger of confusion with the symbol used to describe measure-type potentials.} \cite[Sec.XIII.1]{RS78} and by $\mu_\infty(A)$ the limiting value of this sequence (for operators with a finite discrete spectrum we consider here, the sequence is constant from some index on). Then we introduce the operators
$$
\tilde{H}_{\Gamma,V_{\xi,\lambda}} := H_{\Gamma,V_{\xi,\lambda}} - \mu_\infty(H_{\Gamma,V_{\xi,\lambda}}) \quad\text{and}\quad \tilde{H}_{\Gamma,\alpha_\xi} := H_{\Gamma,\alpha_\xi} - \mu_\infty(H_{\Gamma,\alpha_\xi})
$$
whose essential spectrum threshold is zero by construction; the mentioned result implies that $\tilde{H}_{\Gamma,V_{\xi,\lambda}}$ converges for a fixed $g$ to $\tilde{H}_{\Gamma,\alpha_\xi}$ as $\lambda\to\infty$. At the same time, $\mu_k(\tilde{H}_{\Gamma,\alpha_\xi}) = \epsilon_k(\alpha_\xi)+ \frac14\alpha_\xi^2$ converges by \eqref{leakyasympt} to $\epsilon_k$, the $k$th eigenvalue of $S_\Gamma$. Next we use a simple telescopic estimate,
\begin{equation}
\label{telescopic}
\big| \mu_k(\tilde{H}_{\Gamma,V_{\xi,\lambda}}) - \mu_k(S_\Gamma)\big| \le \big| \mu_k(\tilde{H}_{\Gamma,V_{\xi,\lambda}}) - \mu_k(\tilde{H}_{\Gamma,\alpha_\xi})\big| + \big| \mu_k(\tilde{H}_{\Gamma,\alpha_\xi}) - \mu_k(S_\Gamma)\big|.
\end{equation}
Given an arbitrary $\varepsilon>0$, one can find a $\xi_0>0$ such that the second term on right-hand side of \eqref{telescopic} is smaller that $\frac12\varepsilon$ for all $\xi>\xi_0$, and to such a $\xi$ there is a $\lambda_\xi>0$ with the property that the first term on right-hand side of \eqref{telescopic} is smaller than $\frac12\varepsilon$ for $\lambda>\lambda_\xi$. The map $\xi\mapsto\lambda_\xi$ is obviously monotonous which allows us to define $f_0$ as its pull-back and $g_0(\lambda):= \lambda f(\lambda)$. For any function $g\ge g_0$ we then have $\mu_k(\tilde{H}_{\Gamma,V_{g(\lambda)})}) \to \mu_k(S_\Gamma)$ as $\lambda\to\infty$, and since $\mu_1(S_\Gamma) = \mu_1(S_\Gamma) - \mu_\infty(S_\Gamma)<0$ holds for our non-straight $\Gamma$, the same must be true for $\mu_k(\tilde{H}_{\Gamma,V_{g(\lambda)})})$ if $\lambda$ is large enough, in other words, we conclude $\sigma_\mathrm{disc}(H_{\Gamma,V_{g(\lambda)}})$ must be nonempty for large $\lambda$.
\end{proof}

Note that one could get a better idea of the function $g_0$ looking into the error terms of the tow limits involved, but we will not follow this route here.

\section{A numerical example}\label{s:num}
\setcounter{equation}{0}

While we have been able to establish the weak/strong coupling dichotomy, finer properties are out of reach for an analytical treatment, for instance, it is not easy to find the critical strength needed to produce the discrete spectrum. They can be dealt with numerically as we are going to illustrate here. As an example, we consider the U-shaped channel with a polynomial profile, a multiple of
\begin{equation}
\label{eq:NumericsTransProfile}
    v_{\alpha}(x) \defeq \min\Big\{
        \Big(\frac{|x|-\rho}{a}\Big)^{\alpha} - 1,\: 0 \Big\}.
\end{equation}
where $\Omega \defeq \left(-\rho - a, -\rho+a \right) \bigcup \left(\rho - a, \rho+ a\right) $ and $\alpha\geq 2$ is an even integer. Such a profile obviously tends to a rectangular well as $\alpha\to\infty$ but we work with finite values as it makes the numerics easier.

Naturally, we have to solve the equation in a finite region. We do it for both the Dirichlet and Neumann boundaries which allows us to control the precision using the gap between the two. In computations, we made use of the Spectral Method \cite[Chap.~9]{T00}. Looking for the critical coupling constant of the potential $\lambda v_{\alpha}$, we plot in Figure \ref{fig:CriticalLambda}
\begin{figure}[h]
    \centering
    \includegraphics[width=0.75\textwidth]{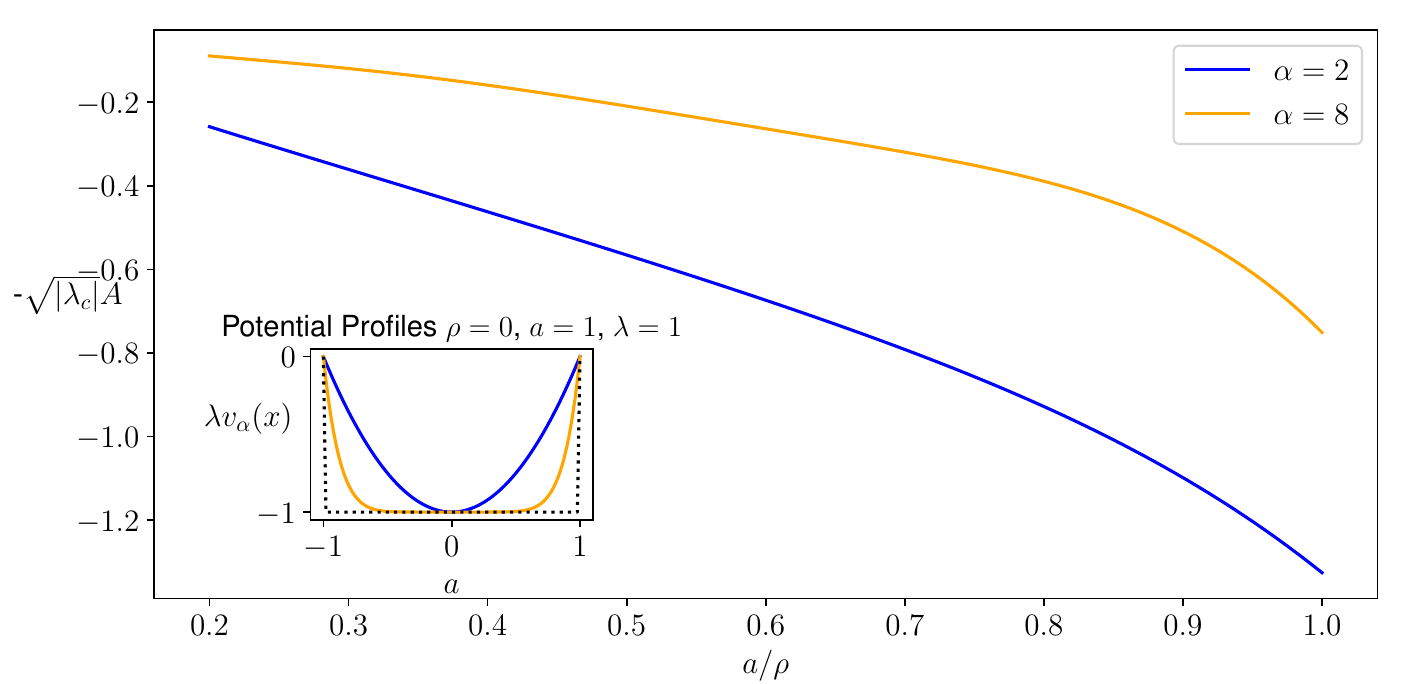}
    \caption{The critical interaction strength \emph{vs.} the relative channel width}
    \label{fig:CriticalLambda}
\end{figure}
the dependence of two dimensionless quantities, $-\sqrt{|\lambda|}A$, where $A:= \frac{1}{\pi} \int_{-a}^a \sqrt{v_{\alpha}(x)}\,\D x$, and $a/\rho$, the relative width of the channel, for two values of $\alpha$. The former is, of course, an upper estimate, asymptotically exact, of the number of bound states minus one in the one-dimensional potential well of that profile. Predictably, the curves are monotonously decreasing since the tunnelling between the parallel channels plays a more prominent role as $a$ increases and a stronger coupling is needed to compete with it. Also the dependence on $\alpha$ makes sense, as the dimensionless `volume' increases as $\alpha$ grows. We also see that for larger $\alpha$ the curve becomes steeper as we approach the point $a=\rho$ as the `residual' barrier is thinner then.

For illustration we also plot in Figure~\ref{BoundState}
\begin{figure}
    \centering
    \includegraphics[width=0.9\textwidth]{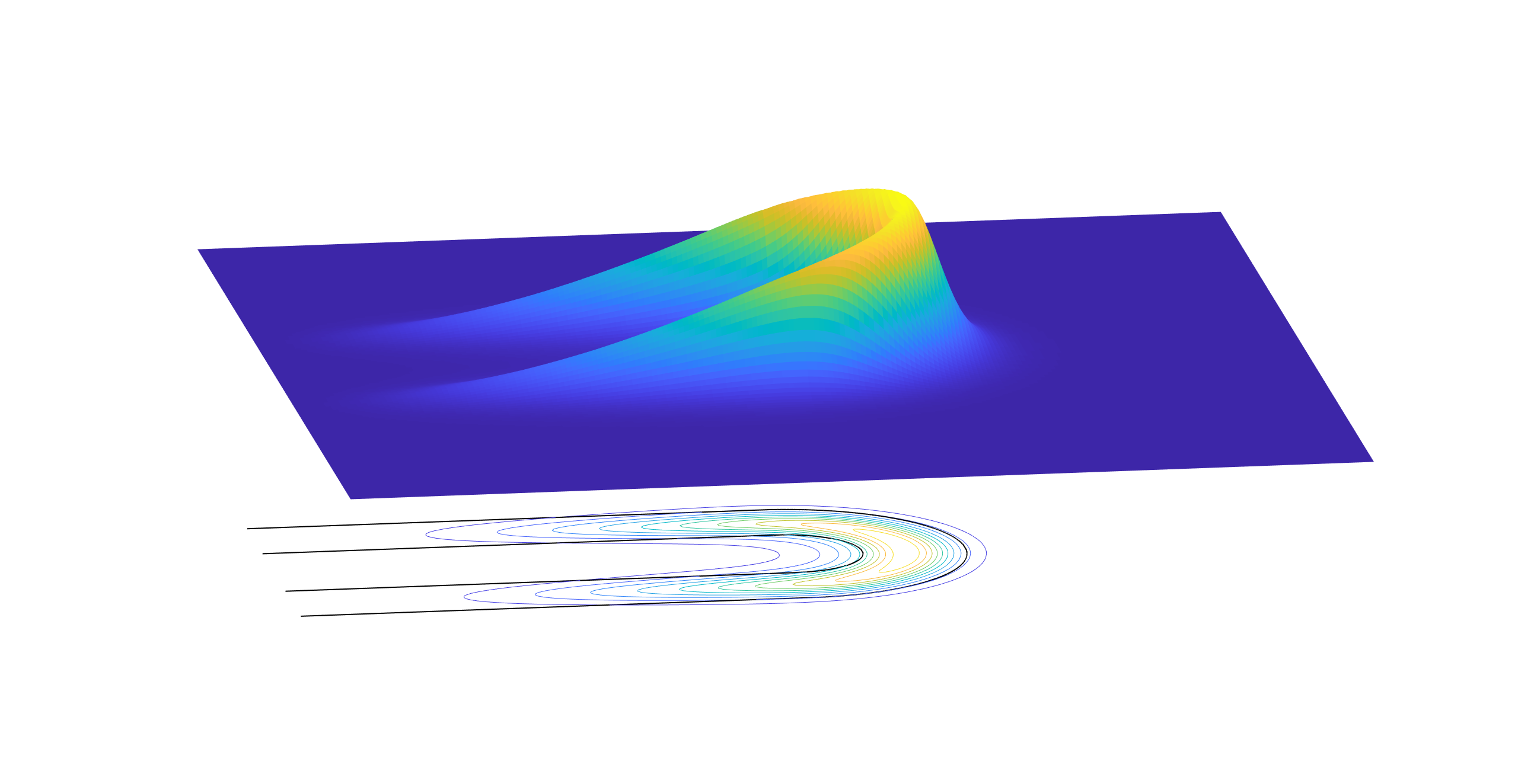}
    \caption{Bound state eigenfunction $\rho=0.25$, $a=0.1$, $\lambda=-225$ and $\alpha=2$.}
    \label{BoundState}
\end{figure}
the eigenfunction in the situation when the coupling is stronger than critical; the thin black line in the picture indicates the potential support.

\section{Concluding remarks} 
\setcounter{equation}{0}

While the shape of the curved part in our result is quite general, modulo the regularity requirement, the guide is supposed to be straight outside a compact. A question arising naturally is what happens if $\Gamma$ approached parallel lines only asymptotically; one expects that the general picture would remain similar as here. Other type of asymptotic behaviors, on the other hand, may lead to different spectral properties. One can ask, for instance, whether there can be geometrically induced eigenfunctions supported away from the `bend' of $\Gamma$; one expects this to happen, say, when the double guide is locally bent or the gap between two channels is locally diminished at an appropriate place.

\subsection*{Data availability statement}

Data are available in the article.

\subsection*{Conflict of interest}

The authors have no conflict of interest.

\subsection*{Acknowledgements}

The work was supported by the Czech Science Foundation within the project 21-07129S by the European Union's Horizon 2020 research and innovation programme under the Marie Sk{\l}odowska-Curie grant agreement No 873071.
Computational resources were provided by the e-INFRA CZ project (ID:90254), supported by the Ministry of Education, Youth and Sports of the Czech Republic.


\end{document}